\documentclass[11pt]{tran-l}
\usepackage[centertags]{amsmath}
\usepackage{amsfonts}
\usepackage{amssymb}
\usepackage{amsthm}
\usepackage{graphics, graphicx}
\usepackage{newlfont}
\usepackage{tabularx}

\theoremstyle{plain}
\newtheorem{thm}{Theorem}[section]

\newtheorem{lem}[thm]{Lemma}
\newtheorem{prop}[thm]{Proposition}
\theoremstyle{definition}
\newtheorem{defn}{Definition}[section]
\theoremstyle{remark}
\newtheorem{rem}{Remark}[section]
\theoremstyle{Example}
\newtheorem{exmp}{Example}[section]
\theoremstyle{Open Question}

\theoremstyle{Conjecture}

\numberwithin{equation}{section}

\begin{document}
\title[
multiset Version of Even-Odd  
]
{
A Multiset Version of Even-Odd Permutations Identity
}
\author{Hossein Teimoori Faal}

%\address{Department of Applied Mathematics
%and Institute for Theoretical Computer Science (ITI), Charles
%University\\
%Malostranske Namesti 25\\ 11800 Praha1\\ Czech Republic \\}

%\email{hossein.teimoori@gmail.com\\iran@kam.mff.cuni.cz}

\maketitle
%%% ----------------------------------------------------------------------

\begin{abstract}
In this paper, 
we give a new bijective proof of 
a multiset analogue of even-odd permutations 
identity. This multiset version is equivalent 
to the original 
coin arrangements lemma 
which is a key combinatorial lemma 
in the Sherman's Proof of a conjecture of Feynman   
about an identity on paths in planar graphs 
related to combinatorial solution of two dimensional 
Ising model in statistical physics.   
\end{abstract}

%%% ----------------------------------------------------------------------
\section{Introduction and Motivation}
The \emph{Ising model}
\cite{Ising1925}
is a theoretical physics model of the \emph{nearest-neighbor}
interactions in a crystal structure.
In the Ising model,
the vertices of a graph $G=(V,E)$ represent \emph{particles} and
the edges describe interactions between pairs of particles.
The most common example
of a two dimensional Ising model 
is a \emph{planar square lattice} where
each particle interacts only with its neighbors.
A factor (weight) 
$J_{ij}$ is assigned to each edge $\{i,j\}$, where this
factor describes the nature of the interaction between particles $i$
and $j$.
A \emph{physical state} of the system is an
assignment of $\sigma_i \in \{+1,-1\}$ to each vertex $i$.
The
\emph{Hamiltonian} (or energy function) of the 
system is
defined as:
$$
H(\sigma)=-\sum_{\{i,j\}\in E}J_{ij}\sigma_i \sigma_j.
$$
The \emph{distribution} of the physical states over
all possible energy levels is encapsulated in the
\emph{partition function}:
$$
Z(\beta,G)=\sum_{\sigma}e^{-\beta H(\sigma)},
$$
where $\beta$ is changed for $\frac{K}{T}$, in which $K$ is a
constant and $T$ is a variable representing the temperature. 
\\ 
Motivated by a generalization of a \emph{cycle} in a graph, 
a set A
of edges is called \emph{even} if each vertex of V is incident with an even number of edges of A.
The generating
function of even subsets denoted by $\mathcal{E}(G,x)$ can be
defined as
$$
\mathcal{E}(G,x)=\sum_{\text {A: A is even}}\prod_{e\in A}x_{e}.
$$
It turns out that the Ising partition function for a
graph $G$ may be expressed in terms of the generating
function of the even sets of the same graph $G$. 
More precisely, we have the following
Van der Waerden's 
formula \cite{Waerden1941}
$$
Z(G,\beta)=2^{\vert V \vert}\prod_{\{i,j\} \in E}\cosh(\beta
J_{ij})\mathcal{E}(G,x)\vert_{x^{J_{ij}}=\tanh(\beta J_{ij})}.
$$ 
Now, let $G =(V,E)$ be
a planar graph embedded in the plane and for each edge $e$
we associate a formal variable $x_e$ which can be seen as
a weight of that edge.
Let $A=(V,A(G))$ be an arbitrary orientation of $G$.
If $e \in E$ then $a_e$ will denote the orientation of $e$ in $A(G)$
and $a^{-1}_e$ will be the reversed orientation to $a_e$.
We put $x_{a_e}=x_{a^{-1}_e}=x_{e}$.
A \emph{circular sequence} $p=v_1,a_1 ,v_2 ,a_2, \ldots,
a_n, (v_{n+1} = v_1)$ is called 
\emph{non-periodic closed walk} if
the following conditions are satisfied: 
$a_{i}\in \{a_e,a^{-1}_e:
e\in E\},a_{i}\neq a^{-1}_{i+1}$ and 
$(a_1,\ldots,a_n)\neq Z^{m}$
for some sequence $Z$ and $m>1$.
We also let $X(p)=\prod_{i=1}^{n}x_{a_i}$. We further 
let $sign(p)
=(-1)^{n(p)}$, where $n(p)$ is a \emph{rotation number} of 
$p$;
i.e., the number of integral revolutions of the
tangent vector. Finally put $W(p) = sign(p)X(p)$.
\\
There is a
natural \emph{equivalence} on non-periodic closed walks; that is,
$p$ is equivalent with \emph{reversed} $p$. Each equivalence
class has two elements and will be denoted by $[p]$. We assume 
$W([p])=W(p)$ and note that this definition is correct since
equivalent walks have the same sign.
\\
The following beautiful formula is   
due to Feynman who conjectured it, 
but did not gave a proof of it. 
It was Sherman who gave a proof based 
on a key \emph{combinatorial} lemma on \emph{coin arrangements} 
\cite{Sherman1960}
.
\begin{thm}[Feynman and Sherman]
Let $G$ be a planar graph. Then
$$
\mathcal{E}(G,x)=\prod[1-W([p])],
$$
where the product is over all equivalence classes of
non-periodic closed walks of G.
\end{thm}
Here is the original statement of the coin arrangement lemma:
Suppose we have a fixed collection of $N$ objects of which
$m_1$ are of one kind, $m_2$ are of second kind, $\ldots$, and
$m_n$ of $n$-th kind. Let $b_{N,k}$ be the number of exhaustive
unordered arrangements of these symbols into $k$ disjoint,
nonempty, circularly ordered sets such that no two circular
orders are the same and none are periodic. Then, we have
$$
\sum_{k=1}^{N}(-1)^{k}b_{N,k}=0,~~~~~~(N> 1).
$$
\\
It is worth to note that when the collection of objects constitute a set of $n$ elements, then 
the numbers $b_{n,k}$ are exactly \emph{Stirling cycle} numbers; that is,
the number of permutations of 
the set 
$
\{
1,2,\ldots, n
\}
$
(or $n$ - permutations)
with exactly $k$ cycles 
in its decompositions into disjoint cycles. 
It is noteworthy that the coin arrangements lemma in this 
particular case, can be 
reformulated as the following well-known
identity in combinatorics of permutations.
\begin{prop}\label{permident1}
[Even-Odd Permutations Identity]
For any integer number $n>1$,
The number of even $n$-permutations 
is the same as the number of odd $n$-permutations.
\end{prop}  
Our main goal here is to formulate a 
\emph{weighted version} of the even-odd permutations 
identity in the multiset setting.

%\section{Multisets and Permutations}
\section{Basic Definitions and Notation}
As Knuth has noted in 
\cite[p.36]{Knuth1981}
, the term 
\emph{multiset} 
was suggested by N.G.de
Bruijn in a private communication to him.
Roughly speaking, a multiset is an unordered collection
of elements in which repetition is allowed. 
\begin{defn}
[
Multiset 
]
Let $\Sigma=\{a_{1}, \ldots, a_{n}\}$ be a 
finite alphabet. 
A multiset $M$ over $\Sigma$ denoted by 
$
[a_{1}^{m_{1}}, a_{2}^{m_{2}}, \ldots, 
a_{n}^{m_{n}}
]
$ 
is a finite collection of elements of $\Sigma$
with $m_{1}$ occurrences of $a_{1}$, $m_{2}$ occurrences 
of $a_{2}$, $\ldots$, and $m_{n}$ occurrences of $a_{n}$.
The number $N = m_{1} + m_{2} + \cdots + m_{n}$ is called 
the cardinality of $M$ and $m_{i}~(1 \leq i \leq n)$ is 
called the multiplicity of the element $a_{i}$. 
\end{defn}

\begin{defn}
[Permutation of a multiset]
Let $M$ be a multiset over a finite alphabet 
$\Sigma$
of cardinality $N$. 
We also let $i \geq N$ be a given integer. 
Then an $i$-permutation of $M$ is defined 
as an ordered arrangement of $i$ elements 
of $M$. In particular, an $N$-permutation of 
$M$ is also called a permutation of $M$.   
\end{defn}
\begin{exmp}
For the alphabet 
$\Sigma = \{
a, b, c
\}$, 
the string 
$\sigma = aabcba$
is a permutation of the multiset $M=[a^{3}, b^{2}, c^{1}]$
.

\end{exmp}
It is worth to note that by a simple counting 
argument, one can obtain that the number of 
permutations of the multiset 
$
M=
[a_{1}^{m_{1}}, a_{2}^{m_{2}}, \ldots, 
a_{n}^{m_{n}}
]
$
of cardinality $N$ is equal to 
$
\frac{N!}{m_{1}!
m_{2}! \cdots m_{n}!
}
$ 
. 
\\
%\section{Combinatorics of Words}
In the rest of this section, we quickly review the basics of 
the combinatorics of words. The reader can consult the 
reference \cite{Lotha1983}.
\\
Let $\Sigma$ be a finite alphabet. The elements of 
$\Sigma$ are called \emph{letters}.  
A finite sequence of elements of $\Sigma$ is called
a \emph{word} ( or string ) over the alphabet $\Sigma$.
An empty sequence of letters is called an \emph{empty} word 
and is denoted by $\lambda$.  
\\
The set of all words over the alphabet 
$\Sigma$ will be denoted by $\Sigma^{\star}$. 
We also denote the set of non-empty words by 
$\Sigma^{+}$.
A word $u$ is called a factor ( resp. a prefix, resp. a suffix)
of a word $w$, if there exists 
words $w_{1}$ and $w_{2}$ such that 
$w = w_{1} u w_{2}$ (resp. $w= u w_{2} $, resp. $w=w_{1}u$).
\\
The $k$-th power of a word $w$ is defined by 
$w^{k} = ww^{k-1}$ with the convention that 
$w^{0} =\lambda$. 
A word $w \in \Sigma^{+}$ is called 
\emph{primitive} if the equation $w = u^{n}$ 
($ u \in \Sigma^{+}$) implies $n=1$. 
Two words $w$ and $u$ are \emph{conjugate} 
if there exist two words $w_{1}$
and $w_{2}$ such that $w = w_{1} w_{2}$
and $u = w_{2} w_{1}$. 
It is easy to see that the conjugacy relation is 
an \emph{equivalence relation}. A \emph{conjugacy class} 
(or \emph{necklace})
is a class of this equivalence relation. 
\\
For an \emph{ordered} alphabet 
$(\Sigma, <)$, the \emph{lexicographic}
order $\trianglelefteq$ on 
$(\Sigma^{\star}, <)$ is defined 
by letting $ w_{1} \trianglelefteq w_{2}$
if 
\begin{itemize}
\item 
$
w_{1} = u w_{2}, \hspace{0.4cm}(u\in \Sigma^{\star})
\hspace{0.4cm} \text{or} 
$

\item 
$
w_{1} = ras,\hspace{0.3cm} w_{2} = rbt\hspace{0.3cm}
a < b,\hspace{0.3cm}\text{for} \hspace{0.2cm}a,b\in\Sigma 
\hspace{0.3cm} \text{and} \hspace{0.2cm} r,s,t \in \Sigma^{\star}.
$
\end{itemize}  
In particular, if 
$w_{1} \trianglelefteq w_{2}$ and 
$w_{1}$ is not a proper prefix of $w_{2}$, 
we write $w_{1} \vartriangleleft w_{2}$.

A word is called a \emph{Lyndon} word 
if it is primitive and the \emph{smallest} word with respect to 
the lexicographic order in it's conjugacy
class.

\begin{exmp}
Let 
$
\Sigma = \{
1,2,3
\} 
$
be an ordered alphabet. Then, $l_{1} = 1123$ and 
$l_{2} = 1223$ are Lyndon words but 
$
l_{3} = 1131
$
is not a Lyndon word. 
\end{exmp}

The following factorization of the words as a 
\emph{non-increasing} 
product of Lyndon words is of fundamental 
importance in the combinatorics of words. 
From now on, we will denote the set of all
Lyndon words by $L$. 
\begin{thm}
[Lyndon Factorization 
]
Any word $w\in \Sigma^{+}$ can be written uniquely as a
non-increasing product of Lyndon words:
$$
w=l_{1}l_{2}\cdots l_{h},~l_{i}\in L,
\hspace{0.5cm}l_{1} 
\trianglerighteq l_{2} \trianglerighteq 
\cdots \trianglerighteq l_{h}.
$$

\end{thm}
One of the important results about the 
characterization of Lydon words 
is the following. 
\begin{prop}
A word $w\in \Sigma^{+}$ is a Lyndon word 
if and only if $w\in \Sigma$ or
$w=rs$ with $r,s\in L$ and 
$r \vartriangleleft s$. 
Moreover, if there
exists a pair $(r, s)$ with $w = rs$ such that $s,w\in
L$ and $s$ of maximal length, then $r\in L$
and $l \vartriangleleft rs \vartriangleleft s$.
\end{prop}

\begin{defn}\label{standfact1}
For $w \in L\backslash \Sigma$ a Lyndon word consisting of more than
a single letter, the pair $(r, s)$ with $w = rs $ such that 
$r,s \in L$ and $s$ of maximal length is called the
\textit{standard factorization} of the Lyndon $w$.
\end{defn}

\section{Multiset Version of Even-Odd Permutations}
In this section, we first briefly review 
basics of the combinatorics of permutations. For 
more detailed introduction see \cite{Bona2010}.
From now on, we will denote
the set $\{1,2,\ldots, n\}$ by $[n]$.
\\
Recall that a \emph{permutation} $\tau $ of a set $[n]$ 
(or simply an $n$-permutation)
is a 
bijective function $\tau:[n] \mapsto [n]$. 
A \emph{one-line} representation of $\tau$
is denoted by 
$
\tau = 
\tau(1) \tau(2) \cdots \tau(n)
$
. 
\\
Recall that from abstract algebra, we know that 
any permutation can be written as a product of 
\emph{disjoint} cycles. 
Hence, a representation of a permutation 
in terms of disjoint cycles is called \emph{cycle} representation. 
\begin{exmp}
Consider the bijective function 
$$
\tau:~[5] \mapsto [5],~~~ \tau(1) = 3,~ \tau(2) = 4,~ 
\tau(3) =1,~ \tau(4) =5,~ \tau(5) = 2. 
$$
A one-line representation of $\tau$ is 
$\tau = 34152$. The cycle representation of $\tau$ 
is equal to $\tau = (13)(245)$. 

\end{exmp}

%\begin{defn}
%[
%Stirling Cycle Number
%]
%The number of permutations of the set $[n]$
%with $k$ cycles 
%is called \emph{Stirling cycle number} and is denoted by
%$s(n,k)$.  
%\end{defn}
The set of all permutations of the set $[n]$ will 
be denoted by $S_{n}$.
\begin{defn}
[
Cycle Index
]
Let $\tau = c_{1}c_{2} \cdots c_{k}$ be the cycle 
representation of the permutation $\tau \in S_{n}$. 
Then, the number $n - k$ is called the 
cycle index of $\tau$ and will be denoted by
$ind_{c}(\tau)$. 
\end{defn}
\begin{defn}
[
Inversion
]
Let 
$\tau \in S_{n}
= \tau(1) \tau(2) \cdots \tau(n)
$ be a permutation. We say that
$(\tau(i), \tau(j))$
is an inversion of $\tau$ if 
$i < j$ implies $\tau(i) > \tau(j)$.
\end{defn}
We will denote the number of inversions 
of a permutation $\tau$ with
$inv(\tau)$. 

We recall the well-known fact due to Cauchy 
\cite{Cauchy1815}
that for 
any permutation $\tau \in S_{n}$,
the parity of $inv(\tau)$ and $ind_{c}(\tau)$ are the same. 
Therefore, we can divide the class of all permutations 
$S_{n}$ into two important subclasses. 
\begin{defn}
[Even - Odd Permutations]
A permutation $\tau = c_{1} c_{2} \cdots c_{k}$ in 
$S_{n}$ 
is called an even (resp. odd) $n$-permutation if 
$ind_{c}(\tau)$ is even (resp. odd). 
\end{defn}
\begin{exmp}
For $n=5$, the permutation 
$\tau = 13524 =(1)(2354)$ has cycle index equal 
to $3$ and hence $\tau$ is an odd permutation, but the cycle index of  
$\tau' = 21354 = (12)(3)(45)$ is $2$ and so the 
permutation $\tau'$ is even. 
\end{exmp}
Considering the above discussions, the coin arrangements
lemma in the case that there exists exactly one coin of each 
type can be restate as follows. 
\begin{prop}
[Set version of coin arrangements]
For any integer $n>1$,
the number of even $n$-permutations is the same as the number 
of odd $n$-permutations. 
\end{prop}

In the rest of this section, we attempt to formulate a
\emph{multiset version}
of the above well-known result in combinatorics of 
permutations. 
\\
For finding the right formulation of the coin arrangement 
lemma for multisets, we have to first replace 
permutations of the set $[n]$ 
with words of length $N$ defined on the multiset 
$M = [
1^{m_1},2^{m_2},\ldots, n^{m_n} 
]
$
of cardinality $N$
. 
The next step is to find the analogue of the 
cyclic decomposition of permutations
into disjoint cycles. 
It seems that 
the \emph{Lyndon factorization} of a word in which
all factors are \emph{distinct} is the suitable candidate. 
Hence, we come up with the following analogue of 
cycle index. 
\begin{defn}
[
Lyndon tuple
]
Let $\Sigma=\{1,2,\ldots, n\}$ be a finite ordered alphabet and
$M = [1^{m_1}, 2^{m_2},\dots, n^{m_n}]$ be 
a multiset 
over $\Sigma$ 
of cardinality $N$. We will call any permutation 
$w = w_{1}w_{2}\cdots w_{N}$
of $M$ an $N$-word over $M$. If $w= l_{1} l_{2} \cdots l_{k}$
is a Lyndon factorization of $w$ in which
$
l_{1} \vartriangleright l_2 \vartriangleright 
\ldots \vartriangleright l_{k}
$
, 
then a tuple 
$tup(w) = (l_k, \ldots , l_2, l_1)$ 
is called a Lyndon tuple of the word $w$ over $M$   
. 
\end{defn}

\begin{rem}
It is noteworthy to mention that a Lyndon tuple 
of a word consists of only distinct Lyndon words. 
\end{rem}

\begin{defn}
[
Lyndon index
]
Let $\Sigma=\{1,2,\ldots, n\}$ be a finite ordered alphabet and 
$M = [1^{m_1}, 2^{m_2},\dots, n^{m_n}]$ be 
a multiset of 
over $\Sigma$
cardinality $N$.
For a $N$-word $w \in \Sigma^{\star}$ over $M$ 
with $tup(w) = (l_{1}, l_{2}, \ldots, l_{k})$ such that 
$l_{1} \vartriangleleft l_{2} 
\vartriangleleft \ldots \vartriangleleft l_{k}$, 
the Lydon index of $w$ denoted by $i_{l}(w)$
is defined to be the number $N-k$.  
\end{defn}

\begin{defn}
[Even-Odd Words]
Let $\Sigma=\{1,2,\ldots, n\}$ be a finite ordered alphabet and 
$M = [1^{m_1}, 2^{m_2},\dots, n^{m_n}]$ be 
a multiset of 
over $\Sigma$ of 
cardinality $N$.
A $N$-word $w \in \Sigma^{\star}
$
over $M$
is said to be even (resp. odd) $N$-word if the Lyndon index 
$i_{l}(w)$ of $w$ is even (resp. odd).  
\end{defn}

\begin{exmp}
For an ordered alphabet 
$
\Sigma =
\{1,2,3\}
$
and a multiset  
$M = [1^{2}, 2, 3]$
,
the $4$-word $w_{1} = 2113 =(2)(113)$ has the Lyndon index 
equals $2$ and hence it is an even $4$-word. But 
the Lyndon index
of 
$w_{2} = 2131=(2)(13)(1)$
is $1$ and so the $4$-word $w_{2}$
is odd. 
\end{exmp}
Thus, we finally get the following reformulation of 
the Sherman's original coin arrangements lemma.
\begin{prop}
[
Multiset version of even-odd permutations identity
]
Let $\Sigma=\{1,2,\ldots, n\}$ be a finite ordered alphabet and 
$M = [1^{m_1}, 2^{m_2},\dots, n^{m_n}]$ be 
a multiset
over $\Sigma$
 of cardinality $N>1$.
Then, 
the number of even $N$-words over $M$ is the same as the number of 
odd $N$-words over $M$. 
\end{prop} 
In the next section, we will give a 
\emph{bijective proof}
a weighted version
of the above coin arrangements lemma.   

\section{Weighted Coin Arrangements Lemma}
In this section, we will first give a weighted 
reformulation of the coin arrangements lemma. 
Then, we present a bijective proof of our 
main result by constructing a weight-preserving 
involution on the set of 
words. But before doing it, for the sake of 
completeness, we present the original proof of 
Sherman based on the so called \emph{Witt identity}
in the context of \emph{combinatorial group theory}
\cite{Hall1968}
.  
\begin{prop}
Let $\Sigma$ be a finite alphabet
of $k$ letters. 
Let $M(m_1,\ldots,m_k)$ be
the number of Lyndon words with $m_1$ occurrences of
$a_1$, $m_2$ occurrences of $a_2$, $\ldots$, $m_k$ occurrences of $a_k$. Let $x_1,\ldots, x_k$
be commuting variables. Then
\begin{equation}
\prod_{m_{1},\ldots,m_{k}\geq 0}(1-x_{1}^{m_{1}}\cdots
x_{k}^{m_{k}})^{M(m_{1}, \ldots, m_{k})}=1-x_{1}-\cdots-x_{k}.
\end{equation}
\end{prop} 
\begin{proof}
By using Lyndon factorization and formal power series 
identities on words, we have 
\begin{eqnarray}
\frac{1}{1-x_{1}-\cdots-x_{k}}&=&\sum_{w \in \lbrace x_1,\ldots, x_k \rbrace^{\star}}\omega=\prod_{l\in
L}\frac{1}{1-l}\nonumber\\
&=&\frac{1}{\prod_{m_{1},\ldots,m_{k}
\geq 0}(1-x_{1}^{m_{1}}\cdots
x_{k}^{m_{k}})^{M(m_{1}, \ldots, m_{k})}}.\nonumber
\end{eqnarray}

\end{proof}
Now, considering the Witt identity, the proof of 
the coin arrangements lemma 
can be simply obtained by equating 
the coefficients of monomials of the same degree
in both sides of the identity. 
\\
To obtain a weighted generalization of the coin 
arrangements lemma, we first associate 
a formal variable $u_{a}$ with each letter $a$ of 
alphabet $\Sigma$ which can be 
viewed as a weight of that letter. 
For any Lyndon word 
$
l = i_{1} i_{2} \cdots i_{h}
$ 
,
we define the weight $wt(l)$
of the Lyndon word $l \in L$ as the 
product of weights of it's letters. That is, 
$
wt(l) = 
u_{i_{1}} u_{i_{2}} \cdots u_{i_{h}}
$ 
. 
The weight of an $N$-word 
$w \in \Sigma^{\star}
$, 
is defined as 
$
wt(w) = \prod_{l \in tup(w)} wt(l)
$
.
From now on, we will denote the set of all
even (resp. odd) $N$-words over $M$ by $E$ 
(resp. $O$).  
Thus, a weighted version of the coin arrangement 
lemma can be read as follows. 

\begin{thm}\label{WCoin1}
[Weighted Coin Arrangements Lemma]
For any multiset $M$ of cardinality $N>1$, 
the weighted sum of even $N$-words over $M$ is the same as 
the weighted sum of odd $N$-words over $M$. In other words, 
we have 
\begin{equation}
\sum_{w \in E} wt(w) = 
\sum_{w \in O} wt(w). 
\end{equation}
\end{thm}
The following lemma is the key in the proof 
of the above theorem. 
\begin{lem}
\begin{description}
\item[i] 
Let $l=rs$ where $r,s\in  L$ with $r \vartriangleleft s$ and
let $ r $ be a single letter Lyndon word. Then, 
$l=(r,s)$ is the standard factorization
of $l$.

\item[ii]
Let $ l=rs$ where 
$r,s\in  L$ $r \vartriangleleft s$ and
let $r=(r_1,s_1)$ be the standard factorization
of $r$ with $r_{1} \vartriangleleft s_{1}$. Then, 
$ l =(r,s)$ is the standard factorization
of $l$.

\end{description}
\end{lem}

\begin{proof}
\begin{description}
\item[i] 
In this case, it is obvious that $s$ is of maximal length. Hence by Definition \ref{standfact1}
, the result is immediate.
\item[ii]
Assume in contrary that $s$ is not of maximal length. Then there exists a Lyndon word
$s'=s'_{1} s$ ($s' \vartriangleleft s$) where $s'$ is of maximal length and
$l=r'_{1}s_{1}'$ with $r'_{1}\in L$. Now if 
$s_{1} \vartriangleleft  s'_{1}$, since
$s'_{1} \vartriangleleft s$ it implies that 
$s_{1} \vartriangleleft s$ which is a contradiction. On the other hand,
since $ r=(r_{1},s_{1})$ is the standard factorization
of $r$, $s'_{1}$ must be a proper right 
factor of $s_{1}$. But we already know that every Lyndon word is
smaller than its any proper right factor. Thus we get 
$s_{1} \vartriangleleft s'_{1}$, which is again a contradiction.
\end{description}
\end{proof}

\begin{proof}[The Proof of Theorem \ref{WCoin1}]
For a given $N$-word $w$ with Lydon tuple 
$
tup(w)= (l_{1}, l_{2}, \ldots, l_{k})
$
,
we call the Lyndon word $l_1$ splittable, if $l_{1}$ is not a 
single letter and the standard factorization 
of $l_{1} = (r_{1},s_{1})$ satisfies
$
s_{1} \vartriangleleft l_{1}
$ 
. 
Now, one of the following cases may happen:
\begin{itemize}
\item 
The Lyndon word $l_{1}$ is splittable. Then, 
a mapping 
$$ 
f:~E \mapsto O,\hspace{0.4cm}w' = f(w),
\hspace{0.4cm}
tup(w') = (r_{1}, s_{1}, l_{2}, \ldots, l_{k})
$$
is a well-defined weight-preserving mapping (because 
$
r_{1} \vartriangleleft s_{1} \vartriangleleft l_{2}
$
and 
$
wt(l_1) =wt(r_1) wt(s_1)
$
)
.

\item 
The Lyndon word $l_{1}$ is not splittable. Then, 
a mapping 
$$ 
g:~O \mapsto E,\hspace{0.4cm}w' = f(w),
\hspace{0.4cm}
tup(w') = (l_{0}, l_{3}, \ldots, l_{k})
$$
with $l_{0} = l_{1} l_{2}$,
is a well-defined weight-preserving mapping (because 
$l_{0} \in L$ and 
$
l_{0} \vartriangleleft l_{2} \vartriangleleft l_{3}
$
with 
$
wt(l_1) =wt(r_1) wt(s_1)
$
)
.

\end{itemize}
Clearly the mappings $f$ and $g$ are inverse of one another. 
Thus, the function $f$
is a wight-preserving bijection form the set of 
even $N$-words to the set odd $N$-words and the conclusion 
immediately follows. 
\end{proof}


\begin{thebibliography}{99}



\bibitem{Ising1925}
E. Ising, 
\textit{Beitrag zur Theorie des Ferromagnetismus}, 
Z. Phys., 31 (1925), pp. 253--258.

\bibitem{Waerden1941}
B. L. Van der Waerden, 
\textit{Die lange Reichweite der regelm¨assigen Atomanordnung in Mischkristallen}, 
Z. Phys., 118 (1941), pp. 473--488.

\bibitem{Sherman1960}
S. Sherman, 
\textit{
 Combinatorial aspects of the Ising model for ferromagnetism. I. A conjecture of
Feynman on paths and graphs
}, 
J. Math. Phys., 
1 (1960), pp. 202--217.

\bibitem{Hall1968}
M. Hall, 
\textit{
The Theory of Groups
}, 
Chelsea Publishing Co., New
York, 1976, Reprinting of the 1968 edition.


\bibitem{Lotha1983}
M. Lothaire, 
\textit{
Combinatorics on Words
}
, Encyclopedia of Mathematics,
vol. 17, Cambridge University Press, Cambridge, 1997, Reprint of the
1983 original.





\bibitem{Bona2010}
M. Bona, 
\textit{
Combinatorics of Permutations
}, 
Discrete Mathematics
and its Applications, Chapman \& Hall/CRC, 2004.

\bibitem{Knuth1981}
D. E. Knuth, 
\textit{
The Art of computer programming
}
, Vol. 2 (Seminumerical Algo-
rithms), Second Edition. Addison-Wesley, Reading Mass. 1981.

\bibitem{Cauchy1815}
A. L. Cauchy,
\textit{
Mémoire sur les fonctions qui ne peuvent obtenir que deux valeurs égale et de signes contraires par suite des transposition opérées entre les variables qu’elles renferment
}
J. de l'École polytechnique, Vol. 10, 29-112; Oeuvres Complètes, ser II. vol. pp. 91--169.



\end{thebibliography}
\end{document}